\documentclass[10pt,a4paper]{article}  

\usepackage[draft]{optional}
\usepackage[british,american]{babel}

\usepackage{mathrsfs}

\usepackage{amsfonts, amsmath, wasysym}

\usepackage{amssymb,amsthm,
paralist
}

\usepackage{
latexsym,
nicefrac,
}

\usepackage[usenames]{color}

\usepackage{url}

\definecolor{darkgreen}{rgb}{0,0.5,0}
\definecolor{darkred}{rgb}{0.7,0,0}
\usepackage[colorlinks, 
citecolor=darkgreen, linkcolor=darkred
]{hyperref}
\usepackage{esint}
\usepackage{bibgerm}
\usepackage[normalem]{ulem}

\theoremstyle{plain}
\newtheorem{lemma}{Lemma}[section]
\newtheorem{thm}[lemma]{Theorem}
\newtheorem{prop}[lemma]{Proposition}
\newtheorem{cor}[lemma]{Corollary}

\theoremstyle{definition}
\newtheorem{defn}[lemma]{Definition}

\newtheorem{rmk}[lemma]{Remark}

\numberwithin{equation}{section}

\newcommand{\pl}[2]{{\frac{\partial #1}{\partial #2}}}

\newcommand{\al}{\alpha}

\newcommand{\ga}{\gamma}

\newcommand{\de}{\delta}
\newcommand{\om}{\omega}
\newcommand{\Om}{\Omega}

\newcommand{\la}{\lambda}
\newcommand{\La}{\Lambda}

\newcommand{\si}{\sigma}

\newcommand{\Si}{\Sigma}
\newcommand{\Tau}{{\cal T}}  
\renewcommand{\th}{\theta}

\newcommand{\ep}{\varepsilon}
\newcommand{\Ups}{\Upsilon}

\newcommand{\R}{\ensuremath{{\mathbb R}}}
\newcommand{\N}{\ensuremath{{\mathbb N}}}

\newcommand{\C}{\ensuremath{{\mathbb C}}}

\newcommand{\weakto}{\rightharpoonup}
\newcommand{\downto}{\downarrow}

\newcommand{\intersect}{\cap}
\newcommand{\union}{\cup}

\DeclareMathOperator{\inj}{inj}

\DeclareMathOperator{\Proj}{Proj}
\newcommand{\norm}[1]{\Vert#1\Vert}  

\def\osc{\mathop{{\mathrm{osc}}}\limits}

\newcommand{\beq}{\begin{equation}}
\newcommand{\eeq}{\end{equation}}
\newcommand{\beqs}{\begin{equation*}}
\newcommand{\eeqs}{\end{equation*}}
\newcommand{\beqa}{\begin{equation}\begin{aligned}}
\newcommand{\eeqa}{\end{aligned}\end{equation}}
\newcommand{\beqas}{\begin{equation*}\begin{aligned}}
\newcommand{\eeqas}{\end{aligned}\end{equation*}}
\newcommand{\brmk}{\begin{rmk}}
\newcommand{\ermk}{\end{rmk}}
\newcommand{\partref}[1]{\hbox{(\csname @roman\endcsname{\ref{#1}})}}
\newcommand{\half}{\frac{1}{2}}

\newcommand*\tr{\mathop{\mathrm{tr}}\nolimits}

\newcommand*\arsinh{\mathop{\mathrm{arsinh}}\nolimits}

\newcommand{\pt}{\partial_t}
\newcommand{\M}{\ensuremath{{\mathcal M}}_{-1}}

\newcommand{\abs}[1]{\vert#1\vert}

\newcommand{\Col}{\mathcal{C}}
\newcommand{\Cyl}{{\mathscr{C}}}
\newcommand{\thin}{\text{-thin}}
\newcommand{\thick}{\text{-thick}}

\title{{\sc
Refined asymptotics of the Teichm\"uller harmonic map flow into general targets
}
\\ 
}
\author{Tobias Huxol, Melanie Rupflin and Peter M. Topping}
\date{\today}

\begin{document}
\maketitle
\begin{abstract}
The Teichm\"uller harmonic map flow is a gradient flow for the harmonic map energy of maps from a closed surface to a general closed Riemannian target manifold of any dimension, where both the map and the domain metric are allowed to evolve.
Given a weak solution of the flow that exists for all time $t\geq 0$, we find a sequence of times $t_i\to\infty$ at which the flow at different scales converges to a collection of branched minimal immersions with no loss of energy.
We do this by developing a compactness theory, establishing no loss of energy, for sequences of almost-minimal maps.

Moreover, we construct an example of a smooth flow for which the image of the limit branched minimal immersions is disconnected. In general, we show that the necks connecting the images of the branched minimal immersions become arbitrarily thin as $i\to\infty$. 
\end{abstract}


\section{Introduction}

The harmonic map flow of Eells and Sampson \cite{ES} in two dimensions is the gradient flow of the 
 harmonic map 
energy $E(u)$ of a map $u$ from a smooth closed oriented Riemannian surface $(M,g)$ of genus $\gamma$ to 
a smooth compact Riemannian manifold $N=(N,G)$ of any dimension. Here the energy $E(u)$ is defined by
$$E(u)=E(u,g):=\half\int_M |du|^2_g dv_g.$$
If we also allow the domain metric $g$ to vary, and restrict $g$ to have fixed constant curvature (and fixed area) 
then we end up with the Teichm\"uller harmonic map flow introduced in \cite{RT}. 
This is the flow given, for fixed parameter $\eta>0$, by 
\begin{equation}
\label{flow}
\pl{u}{t}=\tau_g(u);\qquad \pl{g}{t}=\frac{\eta^2}{4} Re(P_g(\Phi(u,g))),
\end{equation}
where $\tau_g(u):=\tr \nabla du$ is the tension field of $u$, $P_g$ is the $L^2$-orthogonal projection from the space of quadratic differentials on $(M,g)$ onto the space of \emph{holomorphic} quadratic differentials, and
\beq
\Phi(u,g)=(|u_x|^2-|u_y|^2-2i\langle u_x,u_y\rangle)dz^2 \label{def:hopfdiff}
\eeq
is the Hopf differential, written in terms of a local complex coordinate $z=x+iy$. 
See \cite{RT} for further information.
Note that one can read off directly that the curvature of $g$ remains fixed under this flow because \eqref{flow} makes $\pl{g}{t}$ trace-free and divergence-free \cite[Proposition 2.3.9]{RFnotes}.
For $M=S^2$, the only holomorphic quadratic differential is identically zero, so $\pl{g}{t}\equiv 0$ in this case and we are reduced to the classical harmonic map flow.
In general, the flow decreases the energy according to
\beqa
\label{energy-identity}
\frac{dE}{dt}&=-\int_M 
\left[|\tau_g(u)|^2+\left(\frac{\eta}{4}\right)^2 |Re(P_g(\Phi(u,g)))|^2\right]dv_g\\
&=-\|\pt u\|_{L^2}^2-\frac{1}{\eta^2}\|\pt g\|_{L^2}^2.
\eeqa

The global existence theory for this flow, especially \cite{RTnonpositive, RTcontinuation, Ding-Li-Liu} leads naturally to the question of understanding the asymptotics of a given smooth flow as $t\to\infty$, and it is this question that we address in this paper.

Assume for the moment that $M$ has genus $\gamma\geq 2$, so $g$ flows within the space $\M$ of hyperbolic metrics (i.e. Gauss curvature everywhere $-1$).
If we restrict to the case that 
the length $\ell(g(t))$ of the shortest closed geodesic in the domain $(M,g(t))$ is bounded from below by some $\ep>0$ uniformly as $t\to\infty$ 
(which corresponds to no collar degeneration, even as $t\to\infty$) it was proved in 
\cite{RT} that the maps $u(t)$ subconverge, after reparametrisation, to a branched minimal immersion (or a constant map)
with the same action on $\pi_1$ as the initial map $u_0$. The 
same result also applies in the case of $M$ being a torus with 
the length $\ell(g(t))$ of the shortest closed geodesic
bounded away from zero, as follows from the work in \cite{Ding-Li-Liu} when combined with the
Poincar\'e inequality on nondegenerate tori proved in \cite{RT2}.
(See also Remark \ref{rem:Ding-Li-Liu}.)

On the other hand, on surfaces of genus $\gamma\geq 2$ and in the case that $\liminf_{t\to\infty}\ell(g(t))=0$, 
an initial description of the asymptotics of the flow was given in \cite{RTZ}. 
Loosely speaking, it was shown that the surface $(M,g(t))$ will degenerate into finitely many lower genus surfaces, 
with the map $u(t)$ subconverging (modulo bubbling) to branched minimal immersions (or constant maps) on each of these components.
It is convenient for us here to split the main result in \cite{RTZ} 
into an initial result extracting a sequence of times $t_i\to\infty$ at which the flow $(u,g)$ is a sequence of \emph{almost-minimal maps}, and 
a separate compactness result for such sequences.


\begin{defn}
\label{almost_min_def}
Given an oriented closed surface $M$, a closed Riemannian manifold $(N,G)$, and a pair of sequences $u_i:M\to N$ of smooth maps and $g_i$ of metrics on $M$ with fixed constant curvature and fixed area, we say that $(u_i,g_i)$ is a sequence of almost-minimal maps if $E(u_i,g_i)$ is uniformly bounded and
\beq \label{ass:almost-minimal}
\|\tau_{g_i}(u_i)\|_{L^2(M,g_i)}\to 0, \qquad\text{ and }\qquad
\|P_{g_i}(\Phi(u_i,g_i))\|_{L^2(M,g_i)}\to 0.
\eeq
\end{defn}
For a justification of this terminology, note that using the construction in \cite{RT} -- i.e. viewing the harmonic map energy $E$ as a functional on the space of maps and constant curvature metrics modulo diffeomorphisms isotopic to the identity,  equipped with the natural analogue of the Weil-Petersson metric -- the gradient of $E$ at points represented by $(u_i,g_i)$ converges to zero, while critical points are branched minimal immersions (see \cite{RT}).
We are primarily concerned with the case that $M$ has genus at least two, in which case each $g_i$ will be expected to lie in the space $\M$ of hyperbolic metrics.

\begin{prop}
\label{get_ti}
Given an oriented closed surface $M$, a closed Riemannian manifold $(N,G)$, and a smooth flow $(u,g)$ solving \eqref{flow}
for which $\liminf_{t\to\infty}\ell(g(t))=0$, there exists a sequence $t_i\to\infty$ such that
$\lim_{i\to\infty}\ell(g(t_i))=0$ and
$(u(t_i),g(t_i))$ is a sequence of almost-minimal maps.
\end{prop}
This standard observation is proved, for completeness, in Section \ref{sect3}.
The sequence found in this proposition can be analysed with the following result, which is effectively what is proved in \cite{RTZ} (modulo minor adjustments to Section 3 of that paper).

\begin{thm}[Content from \cite{RTZ}]
\label{thm:asymptotics1_new}
Suppose we have an oriented closed surface $M$ of genus $\gamma\geq 2$, a closed Riemannian manifold $(N,G)$, and a sequence $(u_i,g_i)$ of  almost-minimal maps in the sense of Definition \ref{almost_min_def}, for which 
$\lim_{i\to\infty}\ell(g_i)=0$.

Then after passing to a subsequence, there exist 
an integer $1\leq k\leq 3(\gamma-1)$ and a
hyperbolic punctured surface $(\Si,h,c)$ with $2k$ punctures (i.e. a closed Riemann surface $(\hat\Si,\hat c)$, possibly disconnected, that has been punctured $2k$ times and then equipped with a conformal complete hyperbolic metric $h$) such that the following holds.
\begin{enumerate}
\item The surfaces $(M,g_i,c_i)$ converge to the surface $(\Si,h,c)$ by collapsing $k$ simple closed geodesics $\sigma^{j}_{i}$
in the sense of Proposition \ref{Mumford} from the appendix; in particular there is a sequence of diffeomorphisms $ f_i:\Si\to M\setminus \cup_{j=1}^k \sigma^{j}_{i}$
such that $$f_i^*g_i\to h \text{ and } f_i^*c_i\to c \text{ smoothly locally, }$$
where $c_i$ denotes the complex structure of $(M,g_i)$.
\item The maps 
$U_i:=u_i\circ f_i$ converge to a limit $u_\infty$ weakly in $H^1_{loc}(\Si)$ and weakly in $H_{loc}^2(\Si\setminus S)$ as well as strongly in $W_{loc}^{1,p}(\Si\setminus S)$,
$p\in [1,\infty)$,
away from a finite set of points $S\subset \Si$ at which
energy concentrates.
\item 
The limit $u_\infty:\Si\to N$ extends to a smooth branched minimal immersion (or constant map) on each component of the compactification $(\hat\Si,\hat c)$ of $(\Sigma,c)$ obtained by filling in each of the $2k$ punctures.
\end{enumerate}
\end{thm}

In this paper, we take this analysis of the asymptotics 
of $(u_i,g_i)$
and we refine it in several ways. 
First, after passing to a further subsequence, we extract all bubbles that can develop. What is well understood is that we can extract bubbles at each of the points in $S$ (where possibly multiply many bubbles can develop).
Our first task is to isolate a new set of bubbles that are disappearing into the $2k$ punctures found in Theorem \ref{thm:asymptotics1_new}, or equivalently (as we describe below), being lost down the one or more collars that develop in the domain $(M,g_i)$ as $i\to\infty$.

Having extracted the complete set of bubbles, we show that 
the chosen subsequence 
enjoys a no-loss-of-energy property in which the limit 
$\lim_{i\to\infty}E(u_i,g_i)$ 
is shown to be precisely equal to the sum of the energies (or equivalently areas) of the branched minimal immersions found in Theorem \ref{thm:asymptotics1_new} and the new branched minimal immersions represented by the bubbles.
A special case of what we prove below in Theorem \ref{mainthm}, combined with existing theory, is the following result.
(Recall that $\de\thick(M,g)$ consists of all points in $M$ at which the injectivity radius is at least $\de$. Its complement is $\de\thin(M,g)$.)

\begin{thm}
\label{energy_id_thm}
In the setting of Theorem \ref{thm:asymptotics1_new}, there exist two finite collections of nonconstant harmonic maps $\{\om_k\}$ and $\{\Om_j\}$ mapping $S^2 \to N$, such that after passing to a subsequence we have
$$\lim_{\de\downto 0}\lim_{i\to\infty} 
%
E\left(u_i,g_i;\de\thick\,(M,g_i)\right)
=E(u_\infty,h)
+\sum_{k}E(\om_k),$$
and 
$$\lim_{i\to\infty} E(u_i,g_i)=E(u_\infty,h)
+\sum_{k}E(\om_k)+\sum_{j}E(\Om_j).$$
\end{thm}

Showing that no loss of energy occurs in intermediate regions around the bubbles developing at points in $S$ is standard, following in particular the work of Ding and Tian \cite{DT} we describe in a moment, although 
one could also use energy decay estimates of the form we prove
in this paper. 
However, showing that no energy is lost near the $2k$ punctures, away from where the
bubbles develop, is different, and a key ingredient is the Poincar\'e estimate for quadratic differentials discovered in \cite{RT2}, which is applied globally, not locally where the energy is being controlled. 
In this step we exploit the 
smallness of $P_g(\Phi(u,g))$ that holds for almost-minimal maps.
That this is essential is demonstrated by the work of T. Parker \cite{parker} and M. Zhu \cite{zhu_mathz}, which established that energy \emph{can} be lost along `degenerating collars' in general sequences of harmonic maps from degenerating domains.

The following is the foundational compactness result when the domain is fixed,
cf. \cite{struweCMH, DT, QingTian, LinWang, toppingAnnals, finwind}.

\begin{thm}
\label{existing_theory}
Suppose $(\Ups,g_0)$ is a fixed surface, possibly noncompact, possibly incomplete, and let $u_i$ be a sequence of smooth maps into $(N,G)$ from either $(\Ups,g_0)$, or more generally from a sequence of subsets $\Ups_i\subset\Ups$ that exhaust $\Ups$. Suppose that $E(u_i)\leq E_0$ and that
$\|\tau_{g_0}(u_i)\|_{L^2}\to 0$ as $i\to\infty$. Then there is a subsequence for which the following holds true.

There exist a smooth harmonic map $u_\infty:(\Ups,g_0)\to (N,G)$ (possibly constant) and a finite set of points $S\subset\Ups$ such that 
 we have, as $i\to\infty$,
$$u_i\to u_\infty\quad\text{in }W^{2,2}_{loc}(\Ups\backslash S,N),\text{ and}$$
$$u_i\weakto u_\infty\quad\text{weakly in }W^{1,2}_{loc}(\Ups,N).$$
%

At each point in $S$, a bubble tree develops in the following sense.
After picking local isothermal coordinates centred at the given point in $S$, 
there exist a finite number of nonconstant harmonic maps $\om_j:S^2\to (N,G)$, for $j\in \{1,\ldots, J\}$, $J\in\N$ (so-called bubbles) which we view as maps from $\R^2\union \{\infty\}$ via stereographic projection, and sequences of 
numbers $\la_i^j\downto 0$ 
and coordinates $a^j_i\to 0\in\R^2$, 
such that 
$$u_i\left(a^j_i+\la^j_ix\right)\weakto\om_j\qquad\text{weakly in }W_{loc}^{1,2}(\R^2, N).$$
Moreover, we do not count bubbles more than once in the sense that 
\beq \label{eq:bubbles-seperate}\frac{\la_i^j}{\la_i^k}+\frac{\la_i^k}{\la_i^j}
+\frac{|a_i^j-a_i^k|^2}{\la_i^j\la_i^k}\to\infty,\eeq
for each $j,k\in \{1,\ldots, J\}$ with $j\neq k$.

The bubbling has no energy loss in the sense that for each point $x_0\in S$ analysed as above, and each neighbourhood $U\subset\subset \Ups$ of $x_0$ such that $\overline U \intersect S=\{x_0\}$ only,
we have
$$\lim_{i\to\infty} 
E\left(u_i;U\right)=E(u_\infty;U)+\sum_{j}E(\om_j).$$
Moreover, the bubbling enjoys the no-necks property
\beq \label{eq:strong-conv-bubble-tree} u_i(x)-\sum_j \left(\om_j\left(
\frac{x-a^j_i}{\la_i^j}
\right)
-\om_j(\infty)\right)
\to u_\infty(x)\eeq
in $L^\infty(U)$ and $W^{1,2}(U)$ as $i\to\infty$.
\end{thm}

\brmk \label{rmk:first-part-proof}
We note that the proof of the first part of Theorem \ref{energy_id_thm} (virtually) immediately follows 
from Theorems \ref{thm:asymptotics1_new} and \ref{existing_theory}: Away from $S$ we can combine the strong $H^1$- convergence of the maps with the convergence of the metrics. To analyse the maps $U_i=u_i\circ f_i$ near points in $S$ we then apply Theorem \ref{existing_theory} on small geodesic balls $B_{r}^{f_i^*g_i}(p)\subset (\Si, f_i^*g_i)$, 
which are of course isometric to one another provided $r>0$ is chosen sufficiently small as the metrics $g_i$ are all hyperbolic. Finally, the convergence of the metrics allows us to relate
the $\de\thick$ part of $(\Si,f_i^*g_i)$ to the $\de\thick$ part of $(\Sigma,h)$, compare \cite[Lemma A.7]{RTZ}, as well as the geodesic balls $B_{r}^{f_i^*g_i}(p)$ in $(\Si, f_i^*g_i)$ to geodesic balls in $(\Si,h)$.
\ermk

\brmk
\label{collar_intro_rmk}
To do more, we must recall more about the structure of sequences of degenerating hyperbolic metrics, and in particular we need the precise  description of the metrics $g_i$ near to the geodesics $\si_i^j$ of Theorem \ref{thm:asymptotics1_new}, given by the Collar Lemma \ref{lemma:collar} in the appendix.
In particular, for $\de\in (0,\arsinh(1))$ sufficiently small, the $\de\thin$ part of $(M,g_i)$ is isometric to a finite disjoint union of cylinders $\Col^{\de,j}_i:=(-X_\de(\ell^j_i),X_\de(\ell^j_i))\times S^1$ with the metric from Lemma \ref{lemma:collar}; each cylinder
has a geodesic $\si_i^j$ at the centre, with length $\ell^j_i\to 0$ as $i\to\infty$.
These initial observations motivate us to analyse in detail almost-harmonic maps from cylinders.
\ermk

\begin{defn}
\label{bb_def}
When we apply Theorem \ref{existing_theory} in the case that $(\Ups,g_0)=\R\times S^1$ is the cylinder with its standard flat metric, 
then we say that the maps $u_i$ \emph{converge to a bubble branch}, and extract bubbles $\{\Om_j\}$ as follows.
First we add all the bubbles $\{\om_j\}$ to the list $\{\Om_j\}$. In the case that 
$u_\infty:\R\times S^1\to N$ is nonconstant, we view it (via a conformal map of the domain) as a harmonic map from the twice punctured 2-sphere, remove the two singularities
(using the Sacks-Uhlenbeck removable singularity theorem \cite{SU}) to give a smooth nonconstant harmonic map from $S^2$, and add it to the list $\{\Om_j\}$. We say that $u_i$ converges to a \emph{nontrivial} bubble branch if the collection 
$\{\Om_j\}$ is nonempty.
\end{defn}

In the present paper we prove a refinement of the above convergence to a bubble branch. 
To state this result, we shall use the following notations:
For $a<b$, define $\Cyl(a,b):=(a,b)\times S^1$ to be the finite cylinder which will be equipped with the standard flat metric $g_0=ds^2+d\theta^2$ unless specified otherwise. For $\La>0$ we write for short $\Cyl_\Lambda=\Cyl(-\Lambda,\Lambda)$.
Furthermore, given sequences $a_i$ and $b_i$ of real numbers we write $a_i\ll b_i$ if $a_i<b_i$ for all $i\in\N$ and $b_i-a_i\to\infty$ as $i\to\infty$.

\begin{thm}\label{new-thm}
Let $X_i\to \infty$ and let 
 $u_i:\Cyl_{X_i}\to N$ be a sequence of smooth maps with uniformly bounded energy, $E(u_i;\Cyl_{X_i})\leq E_0<\infty$, 
which are almost harmonic in the sense that 
\beq
\label{ass:almost-harmonic}
\norm{\tau_{g_0}(u_i)}_{L^2(\Cyl_{X_i})}\to 0.
\eeq
Then after passing to a subsequence in $i$, there exist
a finite number of sequences $s_i^m$ (for $m\in \{0,\ldots,\bar m\}$, $\bar m\in \N$)
with $-X_i=: s_i^0\ll s_i^1\ll \cdots\ll s_i^{\bar m}:= X_i$   such that the following holds true.
\begin{enumerate}
\item For each $m\in \{1,\ldots,\bar m-1\}$ (if nonempty) the translated maps
$u_i^m(s,\th):=u_i(s+s_i^m,\th)$  converge to a nontrival bubble branch in the sense of 
Definition \ref{bb_def}. 
\item  
The \emph{connecting cylinders} $\Cyl(s_i^{m-1}+\la,s_i^m-\la)$, $\la$ large, are mapped near curves in the sense that
\beq \label{eq:collars-curves_new} \lim_{\la\to\infty}\limsup_{i\to\infty}\sup_{s\in (s_i^{m-1}+\la,s_i^m-\la)}\osc(u_i;\{s\}\times S^1)=0.\eeq
\item If we suppose in addition that the Hopf-differentials tend to zero
\beq
\label{ass:Hopf-0}
\norm{\Phi(u_i)}_{L^1(\Cyl_{X_i})}\to 0
\eeq
then there is no loss of energy on the connecting cylinders 
$\Cyl(s_i^{m-1},s_i^m)$ 
in the sense that for each $m\in \{1,\ldots,\bar m-1\}$
\begin{equation}
\label{noloss_statement_new}
\lim_{\la\to\infty}\limsup_{i\to\infty} E(u_i;
\Cyl(s_i^{m-1}+\la,s_i^m-\la))=0.
\end{equation}
\end{enumerate}
\end{thm}

\begin{defn}
\label{full_bub_branch_def}
In the setting of Theorem \ref{new-thm}, we abbreviate the conclusions of parts 1 and 2 by saying that the maps $u_i$ \emph{converge to a full bubble branch}.  In the case that \eqref{noloss_statement_new} also holds (i.e. the conclusion of part 3) we say that 
the maps $u_i$ \emph{converge to a full bubble branch with no loss of energy}.
\end{defn}


%
Returning to the observations of Remark \ref{collar_intro_rmk},
we note that the length of each of the cylinders 
$\Col^{\de,j}_i$
is converging to infinity, and that any fixed length portion of either end of any of these cylinders will lie within the $\hat\de\thick$ part of $(M,g_i)$ for some small $\hat\de\in (0,\de)$, and thus be captured by the limit $u_\infty$ from Theorem \ref{thm:asymptotics1_new}.
Our main no-loss-of-energy result can therefore be stated as the following result about the limiting behaviour on the middle of the collars, which constitutes our main theorem.

\begin{thm}
\label{mainthm}
In the setting of Theorem \ref{thm:asymptotics1_new}, we fix $j\in \{1,\ldots,k\}$ in order to analyse the $j^{th}$ collar surrounding the geodesic $\si^j_i$. (Now that $j$ is fixed, we drop it as a label for simplicity.)
Thus we consider the collar $\Col(\ell_i)=(-X(\ell_i),X(\ell_i))\times S^1$, with its hyperbolic metric, where $X(\ell_i)\to\infty$. 

Then after passing to a subsequence, the restriction of the maps $u_i$ to the collar $\Col(\ell_i)$ converge to a full bubble branch with no loss of energy in the sense of Definition 
\ref{full_bub_branch_def}.
\end{thm}

We will furthermore analyse almost-minimal maps from degenerating tori in order to obtain the refined asymptotics for the flow \eqref{flow} also in case that the genus of $M$ is one. 

To this end, we first recall that 
 any point in the moduli space of the torus can be represented uniquely as a quotient of $(\C,g_{\text{eucl}})$ with respect to a lattice $\Gamma_{2\pi,a+ib}$ generated by 
$2\pi$ and $a+ib\in\C$ with $-\pi< a \leq \pi$ and $\abs{a+ib}\geq 2\pi$ (with strict inequality if $a<0$), $b>0$.

Given a pair of sequences of maps $u_i:T^2\to N$ and of flat unit area metrics $g_i$ on $T^2$, we let $(a_i,b_i)$ be as above so that $(T^2, g_i)$ is isometric to the quotient of $(\R\times S^1, \frac{1}{2\pi b_i}g_{0})$, with $(s,\theta)$ and 
$(s+ b_i, \theta+a_i)$ identified. We will then  
lift the maps $u_i$ to maps defined on the full cylinder $\R\times S^1$ and 
extend the notion of convergence to a full bubble branch as follows.

\begin{defn}\label{def:fbb-torus}
 Let $(u_i,g_i)$ be as above and assume that  as $i\to \infty$  we have $\ell(g_i)=(\frac{2\pi}{b_i})^\half\to 0$ and 
 $$\norm{\tau(u_i)}_{L^2(\Cyl(0,b_i))}\to 0.$$
 We then say that the maps $u_i$ \textit{converge to a nontrivial full bubble branch} if there exists a 
 sequence $s_i^0$ so that the maps $u_i(s_i^0+\cdot, \cdot)$ converge to a nontrivial bubble branch and 
 so that the restrictions of $u_i(s_i^0+\tfrac{b_i}{2}+\cdot,\cdot)$ to $\Cyl_{\frac{b_i}{2}}$ converge to a full bubble branch in the sense of Definition \ref{full_bub_branch_def}.
\end{defn}

\begin{thm}
 \label{thm:torus}
 Let $(u_i,g_i)$ be a sequence of almost-minimal maps from 
 $M=T^2$ to $N$ 
 for which $\ell(g_i)\to 0$. 
 Then, after passing to a subsequence,  we either obtain that the maps $u_i$ converge to a nontrivial full bubble branch with no loss of energy or we have that 
$E(u_i,g_i)\to 0$ and that the whole torus is mapped near an i-dependent curve in the sense that 
$\limsup_{i\to\infty}\sup_{s\in [0,b_i]}\osc(u_i;\{s\}\times S^1)=0.$
 In particular, we always have that for this subsequence 
\beq
\label{torus_no_loss_of_energy}
\lim_{i\to\infty} E(u_i,g_i)=\sum_j E(\Om_j),
\eeq
where $\{\Om_j\}$ is the collection of all bubbles. 
\end{thm}

As usual, this theorem can be applied to the flow \eqref{flow}, thanks to Proposition \ref{get_ti}.
For a discussion of existing claims concerning no-loss-of-energy for $M=T^2$ 
we refer to Remark \ref{rem:Ding-Li-Liu}.

Both Theorems \ref{mainthm} and \ref{thm:torus} indirectly describe the map $u_i$ on `connecting cylinders' as being close to an $i$-dependent curve, thanks to \eqref{eq:collars-curves_new}.
We are not claiming that this curve has zero length in the limit, as is the case in some similar situations, e.g. for necks in harmonic maps \cite{parker} and the harmonic map flow  from fixed domains \cite{QingTian}. 
We are also not claiming that in some limit the curve should satisfy an equation, for example that it might always be a geodesic as would be the case 
for sequences of harmonic maps from degenerating surfaces, see \cite{Chen-Tian}. 
The following construction can be used to show that these claims would be false in general.

\begin{prop}
\label{nongeod}
Given a closed Riemannian manifold $(N,G)$, a $C^2$ unit-speed curve $\al:[-L/2,L/2]\to N$, and any sequence of degenerating hyperbolic collars $\Cyl_{X_i}$, $X_i\to \infty$, equipped with  their  collar metrics $g_i$ as in the Collar Lemma \ref{lemma:collar},
the maps $u_i:\Cyl_{X_i}\to N$ defined by
$$u_i(s,\th)=\al\left(\frac{Ls}{2X_i}\right)$$
satisfy
$$E(u_i;\Cyl_{X_i}) \leq \frac{CL^2}{X_i}\to 0,$$
$$\|\tau_{g_i}(u_i)\|_{L^2(\Cyl_{X_i},g_i)}^2
\leq C\frac{L^4}{X_i}\to 0$$
and 
$$\|\Phi(u_i,g_i)\|_{L^2(\Cyl_{X_i},g_i)}^2
\leq C\frac{L^4}{X_i}
\to 0$$
\end{prop}

We give the computations in Section \ref{neck_example_sect}.
The proposition can be used to construct a sequence of almost-minimal maps with nontrivial connecting curves: For example, one can take any curve $\al$ as above, and any sequence of hyperbolic metrics $g_i$ with a separating collar degenerating,
and then take the maps $u_i$ to be essentially constant on either side of this one degenerating collar where the map is modelled on that constructed in the proposition.
A slight variation of the construction would show that the $i$-dependent curve need not have a
reasonable limit as $i\to\infty$ in general, whichever subsequence we take, and indeed that its length can converge to infinity as $i\to\infty$.

Finally, we consider the more specific question of what the connecting curves can look like in the case that we are considering the flow \eqref{flow} and 
we have applied Proposition \ref{get_ti} to get a
sequence of almost-minimal maps $u(t_i):(M,g(t_i))\to N$.
One consequence of Theorem \ref{new-thm} is that 
for large $\la$ and very large $i$, the restriction of $u(t_i)$ to the connecting cylinders 
$\Cyl(s_i^{m-1}+\la,s_i^m-\la)$ 
is close in $C^0$ to an $i$-dependent curve $\ga(s)$
connecting the end points
\beq
\label{pplus}
p^{m-1}_+:=\lim_{\la\to\infty}\lim_{i\to\infty}u(t_i)(s^{m-1}_i+\la,\th=0)
\eeq
and 
\beq
\label{pminus}
p^{m}_-:=\lim_{\la\to\infty}\lim_{i\to\infty}u(t_i)(s^{m}_i-\la,\th=0)
\eeq
in the images of branched minimal immersions that we have already found.
(Note that it is not important to take $\th=0$ in these limits. Any sequence $\th_i$ would give the same limits.) 

Now that we have restricted to the particular case in which our theory
is applied to the Teichm\"uller harmonic map flow, 
one might hope to rule out or restrict necks from developing.
However, these necks do exist, and we do not have to have $p^{m-1}_+= p^m_-$,
as we now explain. 

\begin{thm}
\label{neck_thm}
On any oriented closed surface $M$ of genus at least two,
there exists a smooth solution of the Teichm\"uller harmonic map flow 
that develops a nontrivial neck as $t\to\infty$. More precisely, if we extract a sequence of almost-minimal maps 
$(u(t_i),g(t_i))$ as in Proposition \ref{get_ti}, then we can analyse it with Theorem \ref{mainthm}, and after passing to a further subsequence we obtain
\beq
\label{neck_develops}
\lim_{\la\to\infty}\liminf_{i\to\infty}
\osc(u(t_i);\Cyl(s_i^{m-1}+\la,s_i^m-\la))>0
\eeq
for some degenerating collar and some $m\in\{1,\ldots,\bar m\}$.
Moreover, there exist examples for which 
\beq
\label{ps_unequal}
p^{m-1}_+\neq p^m_-,
\eeq
i.e. at least one neck connects distinct points.
\end{thm}
The simplest way of constructing an example as required in the theorem is to arrange that
there can be no nonconstant branched minimal immersions in the limit, while preventing the flow
from being homotopic to a constant map. The flow then forces a collar to degenerate in the limit $t\to\infty$, and maps it to a curve in the target as we describe in Theorem \ref{mainthm}.
The precise construction will be given in Section \ref{neck_example_sect}.
A key ingredient is the regularity theory for flows into nonpositively curved targets developed in \cite{RTnonpositive}.

\begin{rmk}
It would be interesting to prove that in a large class of situations the connecting curves of Theorem \ref{mainthm} and Theorem \ref{thm:torus}, when applied to the Teichm\"uller harmonic map flow, will have a limit, 
and that that limit must be a geodesic.
In the case that $M=T^2$, and under the assumption that the total energy converges to zero as $t\to\infty$, 
Ding-Li-Liu \cite{Ding-Li-Liu} proved that the image of the torus indeed converges to a closed geodesic. 
\end{rmk}

The conclusion of the theory outlined above is a much more refined description of how the flow decomposes an arbitrary map into a collection of branched minimal immersions from lower genus surfaces. 
%

As in \cite{RTnonpositive},
although we state our flow results for compact target manifolds, the proofs extend to some noncompact situations, for 
example when $N$ is noncompact but the image of $u(0)$ lies within the sublevel set of a proper convex function on $N$.

The paper is organised as follows: In the next section we derive bounds on the angular part of the energy of almost harmonic maps on long euclidean cylinders. The main results about almost-minimal maps are then established in Section \ref{sect3}, where we first prove
Theorem \ref{noloss_statement_new}, which then allows us to show Theorem \ref{mainthm}, and as a consequence to complete the proof of Theorem \ref{energy_id_thm}, and to prove Theorem \ref{thm:torus}. In Section \ref{neck_example_sect} we prove the results on the images of the connecting cylinders stated in Proposition \ref{nongeod} and Theorem \ref{neck_thm}. 
In the appendix we finally include the statements of two well-known results for hyperbolic surfaces, the Collar lemma and the Deligne-Mumford compactness theorem, the statements and notations of which are used throughout the paper.

\emph{Acknowledgements:} 
The third author was supported by 
EPSRC grant number EP/K00865X/1.
We thank the referee for detailed commentary and suggestions.

\section{Angular energy decay along cylinders for almost-harmonic maps}

Throughout this section we consider smooth maps $u: \Cyl_\Lambda \to N\hookrightarrow \mathbb{R}^{N_0}$, 
$\Lambda >0$, where $N = (N,G)$ is a compact Riemannian manifold that we isometrically embed in $\mathbb{R}^{N_0}$. 
The tension $\tau$ of $u$ with respect to the 
flat metric $(ds^2+d \theta^2)$ on the cylinder is given by 
	\begin{equation*}
	\tau := u_{\theta \theta} + u_{ss} + A(u) ( u_s, u_s) + A(u) (u_\theta, u_\theta) ,
	\end{equation*}
	where $A(u)$ denotes the second fundamental form of the target $N \hookrightarrow \mathbb{R}^{N_0}$.

Our goal is to prove a decay result for almost-harmonic maps from cylinders, forcing the 
angular energy to be very small on the middle of the cylinder $\Cyl_\Lambda$ when we apply it in the setting of Theorem \ref{mainthm}.
This will be done by first controlling the angular energy, defined in terms of the angular energy on circles
$$\vartheta(s) =\vartheta(u,s):= \int_{\{s\} \times S^1} |u_\theta|^2.$$
The proof of the following lemma is very similar to \cite[Lemma 2.13]{toppingAnnals}, which in turn optimised \cite{LinWang}. Energy decay in such situations arose earlier in \cite{QingTian}, and such results for harmonic functions are classical. More sophisticated decay results were required in \cite{RTnonpositive}.

\begin{lemma}
\label{lem:angularenergy}
For $u:\Cyl_\Lambda\to N$  smooth %
such that $E(u; \Cyl_\Lambda)\leq E_0$,
there exist $\delta > 0$ and $C\in (0,\infty)$ depending only on $N$ and $E_0$, such that if
\begin{equation*}
E(u; \Cyl(s-1,s+2)) < \delta \hspace{0.5em} \text{for every } s
\text{ such that } \hspace{0.5em} \Cyl(s-1,s+2) \subset \Cyl_\Lambda 
\end{equation*}
and
\begin{equation*}
\norm{ \tau}^2_{L^2(\Cyl_\Lambda)} < \delta,
\end{equation*}
then for any $s \in (-\Lambda +1, \Lambda -1)$ we have
\begin{equation}
	\vartheta(s) \leq C e^{|s|-\Lambda} + \int_{-\Lambda}^{\Lambda} e^{-|s-q|} \Tau(q) dq 
		\label{eq:angularestimate}
	\end{equation}
	where 
	$$\Tau(s) :=  \int_{\{s\}\times S^1} |\tau|^2 .$$
Furthermore when $1 < \lambda < \Lambda$, we have the angular energy estimate
\begin{equation}
\label{claim1}
\int_{-\Lambda + \lambda}^{\Lambda - \lambda} \vartheta(s) ds \leq 
C  e^{-\lambda}  + 2\norm{\tau}_{L^2(\Cyl_\Lambda)}^2,
\end{equation}
and thus
\begin{equation}
\label{claim2}
E(u;\Cyl_{\Lambda -\lambda}) \leq	C  e^{-\lambda}  + 2\norm{\tau}_{L^2(\Cyl_\Lambda)}^2 + \frac14\norm{\Phi}_{L^1(\Cyl_{\Lambda -\lambda})}.
\end{equation} 

\end{lemma}

We require a standard `small-energy' estimate, very similar to e.g. \cite[Lemma 2.1]{DT} or \cite[Lemma 2.9]{toppingAnnals}.

\begin{lemma}
	There exist constants $\delta_0 \in (0,1]$ and $C \in (0,\infty)$  depending only on $N$ such that any map $u \in W^{2,2}(\Cyl(-1,2),N))$ which satisfies $E(u; \Cyl(-1,2)) < \delta_0$ must obey the inequality 
	$$ \norm{ u - \bar{u}}_{W^{2,2}(\Cyl(0,1))} \leq C \left( \norm{\nabla u}_{L^2(\Cyl(-1,2))} + \|\tau\|_{L^2(\Cyl(-1,2))}  \right) $$
	where $\bar{u}$ is the average value of u over $\Cyl(-1,2)$.
\label{lem:2}
\end{lemma}

Applying the Sobolev Trace Theorem gives the following (cf. 
\cite{toppingAnnals}):
\begin{cor}
	\label{cor:1} 
	For any map $u \in C^\infty(\Cyl(-1,2),N)$ satisfying  $E(u;\Cyl(-1,2)) < \delta_0$ (where $\delta_0$ originates in Lemma \ref{lem:2}) and for any $s \in (0,1)$, there holds the estimate
	\begin{equation*}
	\int\limits_{\{s\} \times S^1}( |u_\theta|^2 + |u_s|^2 ) \leq C\left( \norm{\nabla u}_{L^2(\Cyl(-1,2))} + \|\tau\|_{L^2(\Cyl(-1,2))}  \right)^2
	\end{equation*}
	with some constant $C$, again only depending on $N$.
\end{cor}

We now establish a differential inequality for $\vartheta(s)$. This is similar to \cite[Lemma 2.1]{LinWang}, but without requiring a bound on $\sup |\nabla u|$. It is proved analogously to \cite[Lemma 2.13]{toppingAnnals}, working on cylinders instead of annuli and considering a general target $N$.

\begin{lemma}
\label{lem:diffineq}
There exists a constant $\delta>0$ depending on $N$ such that 
for $u \in C^\infty(\Cyl(-1,2),N)$ satisfying $E(u;\Cyl(-1,2))<\de$
and $\norm{\tau}^2_{L^2(\Cyl(-1,2))} < \delta$, and for any $s\in (0,1)$,
we have the differential inequality
\begin{equation*}
\vartheta''(s) \geq \vartheta(s) - 2 \int\limits_{\{s\} \times S^1}  |\tau|^2. 
\end{equation*}
\end{lemma}

\begin{proof}[Proof of Lemma \ref{lem:diffineq}]

	From the proof of \cite[Lemma 3.7]{RTnonpositive} we have the expression
	\begin{equation}
	\label{eq:thetadashdash}
	\vartheta''(s) = 2 \int_{ \{s\} \times S^1} |u_{s \theta}|^2 + |u_{\theta \theta}|^2 - u_{\theta \theta} \cdot \tau + u_{\theta \theta} \left[ A(u) ( u_s, u_s) + A(u) (u_\theta, u_\theta) \right]. 
	\end{equation}
	We can estimate the penultimate term as in \cite{RTnonpositive} using integration by parts and Young's inequality:
\begin{equation*}
\begin{aligned}
\left| 2\int u_{\theta \theta} \cdot \left[ A(u)(u_s,u_s) \right] \right| 
&\leq 	C \int |u_\theta|^2 |u_s|^2 + |u_{s \theta}| |u_s| |u_\theta|  \\ 
&\leq \int |u_{s \theta}|^2 +
C \int  |u_\theta|^2 |u_s|^2,
\end{aligned}
\end{equation*}
while the final term of \eqref{eq:thetadashdash} requires just Young's inequality:
\begin{equation*}
\left| 2\int u_{\theta \theta} \cdot \left[ A(u)(u_\theta,u_\theta) \right] \right| 
\leq C \int |u_{\theta \theta}| |u_\theta|^2 
\leq \frac{1}{4}\int |u_{\theta \theta}|^2 + C \int |u_\theta|^4
\end{equation*}
where $C$ is a constant only depending on $N$, that is revised at each step. Summing gives
	\begin{align}
	\begin{split}
		2\left| \int_{\{s\} \times S^1} u_{\theta \theta} \cdot \left[ A(u) ( u_s, u_s) + A(u) (u_\theta, u_\theta) \right] \right| \leq C &\int_{\{s\} \times S^1} |u_\theta|^2 ( |u_s|^2 + |u_\theta|^2 ) \\ 
		+ &\int_{\{s\} \times S^1} |u_{s\theta}|^2 + \frac{1}{4} |u_{\theta \theta}|^2 \label{eq:ineq1}.
	\end{split}
	\end{align}

To apply Corollary \ref{cor:1}, we can ask that $\delta < \delta_0$, 
and thus handle the first term on the right-hand side as follows:
\begin{align*}
	\int_{\{s\} \times S^1} &|u_\theta|^2 ( |u_s|^2 + |u_\theta|^2 ) \\
	&\leq C \sup_{\{s\}\times S^1} |u_\theta|^2 \left( \norm{ \nabla u}_{L^2(\Cyl(-1,2))} +\|\tau\|_{L^2(\Cyl(-1,2))} \right)^2  \\
	&\leq C \left( \int_{\{s\} \times S^1} |u_{\theta \theta} |^2 \right)  \delta 
\end{align*}
and thus for $\delta$ sufficiently small, depending on $N$, we can improve \eqref{eq:ineq1} to
\begin{equation}
		2\left| \int_{\{s\} \times S^1} u_{\theta \theta} \cdot \left[ A(u) ( u_s, u_s) + A(u) (u_\theta, u_\theta) \right] \right| \leq 
\int_{\{s\} \times S^1} |u_{s\theta}|^2 + \frac{1}{2} |u_{\theta \theta}|^2 \label{eq:ineq_new}.
\end{equation}

It remains to estimate the inner product of $u_{\theta \theta}$ with the tension in \eqref{eq:thetadashdash}.
By Young's inequality, we have
	\begin{align}
	\left| 2 \int_{\{s\} \times S^1} u_{\theta \theta} \cdot \tau \right| 
	\leq \frac{1}{2} \int_{\{s\} \times S^1}  |u_{\theta \theta} |^2 + 2 \int_{\{s\} \times S^1}|\tau|^2,  \label{eq:ineq3}
	\end{align}
and so combining \eqref{eq:ineq_new} and \eqref{eq:ineq3} with \eqref{eq:thetadashdash} gives the estimate
	\begin{align*}
	\begin{split}
		\vartheta''(s) \geq \, 2& \int_{ \{s\} \times S^1} |u_{s \theta}|^2 + |u_{\theta \theta}|^2  \\
				         -& \left( \frac{1}{2} \int_{\{s\} \times S^1}  |u_{\theta \theta} |^2  + 2  \int_{\{s\} \times S^1}|\tau|^2  +  \int_{\{s\} \times S^1}  |u_{s\theta}|^2 + \frac{1}{2} |u_{\theta \theta}|^2   \right)
	\end{split}
\\	\geq  & \int_{ \{s\} \times S^1} |u_{\theta \theta}|^2 - 2 \int_{\{s\} \times S^1}|\tau|^2  \\
	\geq & \int_{ \{s\} \times S^1} |u_{\theta }|^2 - 2 \int_{\{s\} \times S^1}|\tau|^2 
	\end{align*}
		by Wirtinger's inequality.
\end{proof}	

Lemma \ref{lem:diffineq} can be applied all along a long cylinder $\Cyl_\Lambda$ as arising in Lemma \ref{lem:angularenergy}, and we can analyse the resulting differential inequality as in the next lemma to deduce bounds on $\vartheta$.
\begin{lemma}
	\label{lem:ode}
	Consider a smooth function $f:[S_1,S_2] \to \mathbb{R}$ satisfying the inequality
	\begin{equation}\label{eq:ODE}
	 f''(s)- f(s) \geq - 2 \Tau(s),
	\end{equation}
	with given boundary values $f(S_1)$, $f(S_2) \in [0,2E_0]$, 
	and $\Tau:[S_1,S_2] \to [0,\infty)$ smooth.
	Then
	$$ f(s) \leq 2E_0 \left( e^{s-S_2} + e^{S_1-s}\right) + \int\limits_{S_1}^{S_2} e^{-|s-q|} \Tau(q) dq $$
	for $s \in (S_1+1,S_2-1)$.
\end{lemma}
\begin{proof}
	Recall that in the equality case for \eqref{eq:ODE} a solution $\tilde f$ can be written explicitly as %
	\begin{equation*}
		\tilde f(s) := A e^s + B e^{-s} + \int_{S_1}^{S_2} e^{-|s-q|} \Tau(q) dq, \qquad A,B\in\R.
	\end{equation*}
	We then select $A=2E_0e^{-S_2}$ and $B=2E_0e^{S_1}$ to obtain such a solution for which $\tilde f(S_1)\geq 2E_0\geq f(S_1)$ 
	and $\tilde f(S_2)\geq 2E_0\geq f(S_2)$. The maximum principle implies $\tilde f\geq f$ and thus the claim. 
\end{proof}

We now apply the estimate from Lemma \ref{lem:ode} to establish decay of angular energy.

\begin{proof}[Proof of Lemma \ref{lem:angularenergy}.]
First note that we may assume that $\Lambda\geq 1$, otherwise the lemma is vacuous.
By definition of $\vartheta$, and the upper bound on the total energy, 
we have
	\begin{equation*}
		\int_{-\Lambda}^{\Lambda} \vartheta(q) dq \leq 2E(u; \Cyl_\Lambda) \leq 2E_0 .
	\end{equation*}
We choose $\delta>0$ smaller than both the $\delta$ of Lemma \ref{lem:diffineq} and the $\delta_0$ in Corollary \ref{cor:1}.

	From the above we obtain that there must exist $S_1 \in [-\Lambda,-\Lambda+1)$ and $S_2 \in (\Lambda-1,\Lambda]$ such that $\vartheta(S_1) \leq 2E_0$ and $\vartheta(S_2) \leq 2E_0$. 
As before, we write
	\begin{align*}
		\Tau(s) := \int\limits_{\{s\}\times S^1} |\tau|^2.
	\end{align*}
Then by Lemma \ref{lem:diffineq}, $\vartheta$ satisfies $\vartheta'' - \vartheta \geq -2\Tau$ on $[S_1,S_2]$. Applying Lemma \ref{lem:ode} then gives the first conclusion \eqref{eq:angularestimate} of Lemma \ref{lem:angularenergy}.

To prove the energy estimate \eqref{claim1} we integrate \eqref{eq:angularestimate} and obtain
	\begin{equation}
		\int_{-\Lambda + \lambda}^{\Lambda - \lambda}   \vartheta(s) ds \leq 
		C  \int_{-\Lambda + \lambda}^{\Lambda - \lambda}   e^{|s|-\Lambda} ds
			+ \int_{-\Lambda + \lambda}^{\Lambda - \lambda}  \int_{-\Lambda}^{\Lambda} e^{-|s-q|} 
			\Tau(q) dq ds .\label{est:theta-in-proof}
	\end{equation}
	
	We can calculate the first integral on the right-hand side explicitly:
	\begin{equation}
		\int_{-\Lambda + \lambda}^{\Lambda - \lambda}   e^{|s|-\Lambda} ds = 2 \left( e^{-\lambda} 
											   - e^{-\Lambda}\right)
											\leq 2 e^{-\lambda}.
\label{eq:exponentialbound}
	\end{equation}	
In the second integral we change the order of integration
\begin{equation*}
\int_{-\Lambda + \lambda}^{\Lambda - \lambda}   \int_{-\Lambda}^{\Lambda} e^{-|s-q|} \Tau(q) dq ds 
= \int_{-\Lambda}^{\Lambda} \Tau(q) \int_{-\Lambda + \lambda}^{\Lambda -	\lambda}   e^{-|s-q|}  ds dq, 
\end{equation*}	
and estimate
$$\int_{-\Lambda + \lambda}^{\Lambda - \lambda}   e^{-|s-q|}  ds
\leq \int_{-\infty}^{\infty}   e^{-|s-q|}  ds=2,$$
to find that 
	\begin{equation*}
		\int_{-\Lambda + \lambda}^{\Lambda - \lambda}   \int_{-\Lambda}^{\Lambda} e^{-|s-q|} 
			\Tau(q) dq ds \leq 2 \int_{-\Lambda}^{\Lambda} \Tau(q) dq  \leq 2 \norm{ \tau}_
														{L^2(\Cyl_\Lambda) }^2.
\end{equation*}
Together with \eqref{est:theta-in-proof} and \eqref{eq:exponentialbound} this implies 
claim \eqref{claim1}. 
To prove \eqref{claim2} we compute
%
\begin{equation}
\label{eq:energy-loss-without-norm}
E(u;\Cyl_{\Lambda -\lambda}) = 
\half\int_{\Cyl_{\Lambda -\lambda}} ( |u_\theta|^2 + |u_s|^2) \,d\theta ds =
\frac12\int_{\Cyl_{\Lambda -\lambda}} (\abs{u_s}^2-\abs{u_\theta}^2) \,d\theta ds+
\int_{\Cyl_{\Lambda -\lambda}} |u_\theta|^2 d\theta ds
	\end{equation}
so by \eqref{claim1} and the definition \eqref{def:hopfdiff} of $\Phi$, see also \eqref{eq:L1-Phi-scaling}, we have
\begin{equation*}
E(u;\Cyl_{\Lambda -\lambda}) \leq	C  e^{-\lambda}  + 2\norm{\tau}_{L^2(\Cyl_\Lambda)}^2 + \frac14\norm{\Phi}_{L^1(\Cyl_{\Lambda -\lambda})}.
\end{equation*}
\end{proof}

%
\section{Proofs of the main theorems; convergence to full bubble branches}
\label{sect3}

Our main initial objective in this section is to prove Theorem \ref{new-thm}, giving convergence of almost-harmonic maps to full bubbles branches.
This will then be combined with the Poincar\'e estimate for quadratic differentials of \cite{RT2} to give
Theorem \ref{mainthm}.
Theorem \ref{new-thm} will also lead us to 
a proof of the analogous result for tori, Theorem \ref{thm:torus}.
But we begin with a proof of the almost-standard Proposition \ref{get_ti}.

\begin{proof}[Proof of Proposition \ref{get_ti}]
It is a standard idea (see e.g. \cite[section 3]{RTZ}) that by integrating \eqref{energy-identity} from $t=0$ to $t=\infty$, and appealing to the boundedness of the energy, one can see that it is possible to extract a sequence of times $t_i\to\infty$ at which $(u(t_i),g(t_i))$ is a sequence of almost-minimal maps.
Note here that by a simple computation, we have 
$$\|P_{g_i}(\Phi(u_i,g_i))\|_{L^2(M,g_i)}^2=2\|Re(P_{g_i}(\Phi(u_i,g_i)))\|_{L^2(M,g_i)}^2.$$
In order to modify this argument to ensure that $\ell(g(t_i))\to 0$ as $i\to\infty$ (by virtue of the hypothesis $\liminf_{t\to\infty}\ell(g(t))=0$) we must argue that it is impossible for $\ell(g(t))$ to spend almost all of the time away from zero, but drop quickly and occasionally down near zero. 
When the genus of $M$ is at least $2$, this is ruled out by \cite[Lemma 2.3]{RTnonpositive}, which implies in particular that
$$\bigg|\frac{d}{dt}\ell(g(t))\bigg|\leq C,$$ 
whenever $\ell<2\arsinh(1)$,
where $C$ depends on an upper bound for the energy, e.g. $E_0:=E(u(0),g(0))$, the genus of $M$ and the coupling constant $\eta$. 
Alternatively, for any genus, we know that for $\ell(g(t))$ to oscillate in this bad way would require $g(t)$ to move each time a definite distance in Teichm\"uller space equipped with the Weil-Petersson metric (see e.g. \cite[\S 4.4]{RTnonpositive}) but $g(t)$ represents a uniformly $C^{0,1/2}$ path in Teichm\"uller space because the distance moved between times $0<t_1<t_2<\infty$
is proportional to 
$$\int_{t_1}^{t_2}\|\pt g\|_{L^2}dt\leq \eta\int_{t_1}^{t_2}\left(-\frac{dE}{dt}\right)^\half dt
\leq \eta(t_2-t_1)^{1/2}E_0^{1/2},$$
by \eqref{energy-identity}
\end{proof}

\begin{proof}[Proof of Theorem \ref{new-thm}]
Let $u_i:\Cyl_{X_i}\to N$ be a sequence of smooth almost harmonic maps as considered in Theorem \ref{new-thm}. 
The first task is to construct sequences $s_i^m$ as in the statement of the theorem. We would like to apply \eqref{claim2} on the regions $\Cyl(s_i^{m-1} + \lambda, s_i^{m} - \lambda)$ to the maps $u_i$ for large $i$
so we let $\delta > 0$ be as in Lemma \ref{lem:angularenergy}, which will be independent of $i$, of course.

We proceed to construct auxiliary sequences $\hat{s}_i^m$, where $m \in \{0, \dots, \hat{m}+1\}$ for some $\hat{m} \geq 0$. 
For each $i$, consider the
overlapping chunks of length 3 of the form $(k-1,k+2)\times S^1\subset \Cyl_{X_i}$ for  $k \in \mathbb{Z}$, i.e. for integral $k$ such that $-X_i < k - 1 < k+2 < X_i$. These chunks cover 
$\Cyl_{X_i}$ except 
possibly for cylinders of length no more than $1$ at the ends. 

For each $i$, we initially choose the numbers $\hat{s}_i^m$, for $m=1,2,\ldots,m_i$,
to be the increasing sequence of integers so that $(\hat{s}_i^m-1,\hat{s}_i^m+2)\times S^1$ are precisely the chunks above that have energy at least $\frac{\delta}{2}$.

Note that by the bound on the total energy, there is a uniform bound on the number $m_i$ of such chunks, 
depending only on $N$ and $E_0$. Finally, we add in $\hat{s}_i^0 = -X_i$ and $\hat{s}_i^{m_i + 1} = X_i.$ By passing to a subsequence of the $u_i$ we can assume that for each $i$, we have the same number of sequence elements $\hat{s}_i^m$, i.e. $m_i=\hat m$ for each $i$. Note also that for any region $(s-1,s+2)\times S^1\subset\Cyl_{X_i}$ with energy at least $\delta$ there is some associated overlapping integer chunk $(k-1,k+2)\times S^1\subset\Cyl_{X_i}$ of energy at least $\frac{\delta}{2}$ which is assigned a label in the above construction, except possibly for regions very close to the ends of the cylinder in the sense that $s-1<-X_i+1$ or $s+2>X_i-1$.

From this auxiliary sequence we form $s_i^m$. Set 
$s_i^0 = -X_i=\hat{s}_i^0$, and consider the difference $\hat{s}_i^1-s_i^0$. If this has a subsequence converging to infinity, pass to that subsequence and take $s_i^1 = \hat{s}_i^1$; 
if not, discard $\hat{s}_i^1$. Proceed iteratively to define $s_i^m$ (i.e. $s_i^2$ is the next $\hat{s}_i^m$ such that the respective difference $\hat{s}_i^m-s^1_i$ diverges for some subsequence, after having passed to that subsequence). This process will terminate with the selection of 
$s_i^{\bar m}$,
for some $\bar{m}$. Whatever sequence $s_i^{\bar{m}}$ was chosen, redefine it as $s_i^{\bar{m}} = X_i$, which can only change it by an amount that is uniformly bounded in $i$. This finishes the construction. 

Note that in particular, for $m\in\{1,2,\ldots, \bar{m}-1\}$,
we can carry out the shift $u_i^m(s, \theta) := u_i(s + s^m_i, \theta)$ 
and obtain convergence to a nontrivial bubble branch for each $m$ with associated bubbles $\{\Om_j\}$ after passing to a further subsequence (recall Definition \ref{bb_def}) which completes the proof of Part 1 of the theorem. 

Next we consider the connecting cylinders $\Cyl(s_i^{m-1}+\lambda,s_i^m-\lambda)$
for $m\in\{1,2,\ldots, \bar{m}\}$ and large $\lambda$. By construction, there exists a constant $K>0$ such that 
$$
E( u_i; \Cyl(s-1,s+2))<\delta \text{ for } s \in (s_i^{m-1}+K+1,s_i^m-K-2), 
$$
for sufficiently large $i$ (otherwise we would not have discarded the respective $\hat{s}_i^m$). Now let 
\begin{equation*}
	\Lambda_i^m = \frac{(s_i^m - K) - (s_i^{m-1} + K)}{2}= \frac{s_i^m - s_i^{m-1} - 2K}{2}.
\end{equation*}
By translation we can consider $u_i$ on $\Cyl_{\Lambda_i^m}$. We denote the shifted maps as 
$$\hat u_i^m(s,\theta)= u_i\left(s+\frac{s_i^m+s_i^{m-1}}{2}, \theta\right).$$ 
For each $\la>1$,
the estimate \eqref{claim2} from Lemma \ref{lem:angularenergy}
applies (as in particular we have no concentration of energy) for sufficiently large $i$, giving
\beqa
	E(u_i;\Cyl(s^{m-1}_i  + K + \lambda, s^{m}_i  - K - \lambda)) =&\, E(\hat u_i^m;\Cyl_{\Lambda_i^m -\lambda}) 
	\\
	\leq & \,
	\label{aaa}
	C e^{-\lambda}  + 2\norm{\tau(\hat u_i^m)}_{L^2(\Cyl_{\Lambda_i^m})}^2  + \frac{1}{4}\norm{\Phi(\hat u_i^m)}_{L^1(\Cyl_{\Lambda_i^m -\lambda})}.
\nonumber
\eeqa
Taking the limit $i\to\infty$, and using 
the assumption \eqref{ass:almost-harmonic}
we find that
\begin{equation*}
\limsup_{i \to \infty} E(u_i;\Cyl(s^{m-1}_i  + K + \lambda, s^{m}_i  - K - \lambda)) \leq C e^{-\lambda}+\frac14\limsup_{i\to \infty }\norm{\Phi(u_i)}_{L^1(\Cyl_{X_i})}.
\end{equation*}
Letting $\lambda \to \infty$ proves that the `no-loss-of-energy' claim \eqref{noloss_statement_new} 
holds true provided the maps satisfy the additional assumption \eqref{ass:Hopf-0}, which completes the proof of Part 3 of the theorem. We remark that this last step is the only part of the proof where \eqref{ass:Hopf-0} is used.

Finally we consider the quantity 
\begin{equation*}
	\sup_{s \in (s_i^{m-1}+ K+ \lambda, s_i^m - K - \lambda)} \text{osc}(u_i, \{s\} \times S^1),
\end{equation*} again for $1 < \lambda < \Lambda_i^m$. After applying the same shift as above, this is equivalent to
\begin{equation*}
	\sup_{s \in (-\Lambda^m_i + \lambda, \Lambda^m_i - \lambda)} \text{osc}(\hat u_i^m, \{s\} \times S^1).
\end{equation*} 
On each circle $\{s\} \times S^1$ we have a bound on the (shifted) angular energy
$\vartheta_i^m(s) := \vartheta(\hat u_i^m,s)$  from Lemma \ref{lem:angularenergy} (at least for sufficiently large $i$) for $s \in (-\Lambda^m_i + \lambda, \Lambda^m_i - \lambda )$ given by
\begin{equation*}
	\vartheta_i^m(s) \leq C e^{|s|-\Lambda_i^m} + \int_{-\Lambda_i^m}^{\Lambda_i^m} e^{-|s-q|} \Tau(q) dq  \leq C e^{|s|-\Lambda_i^m} + \int_{-\Lambda_i^m}^{\Lambda_i^m} \Tau(q) dq,
\end{equation*}
and thus we have
\begin{equation*}
\begin{aligned}
	\sup_{s \in (-\Lambda^m_i + \lambda, \Lambda^m_i - \lambda)} \vartheta_i^m(s) &\leq C e^{-\lambda} + \int_{-\Lambda_i^m}^{\Lambda^m_i} \Tau(q) dq \\
	&\leq C e^{-\lambda} + \norm{\tau}_{L^2(\Cyl_{\Lambda^m_i})}^2.
\end{aligned}
\end{equation*}
Taking limits, and using once more \eqref{ass:almost-harmonic}, gives
\begin{equation*}
\limsup_{i \to \infty} \sup_{s \in (-\Lambda^m_i + \lambda, \Lambda^m_i - \lambda)} \vartheta_i^m(s) \leq C e^{-\lambda},
\end{equation*}
and then 
\begin{equation*}
\lim_{\la\to\infty}\limsup_{i \to \infty} \sup_{s \in (-\Lambda^m_i + \lambda, \Lambda^m_i - \lambda)} \vartheta_i^m(s) =0.
\end{equation*}
We conclude by observing that by the fundamental theorem of calculus and Cauchy-Schwarz, we can control the oscillation of 
$\hat u_i^m$ on a circle in terms of the angular energy on that circle, by
\begin{equation*}
\left[\osc(\hat u_i^m, \{s\} \times S^1)\right]^2 \leq 2\pi\vartheta_i^m(s),
\end{equation*}
which implies the  oscillation bound for $u_i$ claimed as Part 2 of the theorem.
\end{proof}

The principle of the above proof also allows us to estimate the loss of energy on the connecting cylinders for sequences  $u_i$ of almost harmonic maps that do not satisfy \eqref{ass:Hopf-0},
as we shall require in the case $M=T^2$.
 
 \begin{lemma}\label{cor:new}
  In the setting of Theorem \ref{new-thm}
  we have that for each $m=2,\ldots, \bar m-1$
\beq \label{eq:loss-energy}
\lim_{\la\to\infty}\limsup_{i\to\infty} 
\bigg|E(u_i;
\Cyl(s_i^{m-1}+\la,s_i^m-\la)) -\frac12\int_{\Cyl(s_i^{m-1},s_i^m)} (\abs{\partial_s u_i}^2-\abs{\partial_\theta u_i}^2) \,d\theta ds \bigg|=0
\eeq
while for $m=1\neq \bar m$, respectively $m=\bar m\neq 1$, respectively $m=1=\bar m$ the analogue of \eqref{eq:loss-energy}    
holds true with the integral over 
$\Cyl(s_i^{0}+\lambda,s_i^1)$,
respectively
$\Cyl(s_i^{\bar m-1},s_i^{\bar m}-\lambda)$, respectively $\Cyl(s_i^{0}+\lambda,s_i^{\bar m}-\lambda)$.
\end{lemma}
To prove this lemma we will use
\begin{rmk}\label{rem:conv-Hopf}
Recalling that bubble maps 
have vanishing Hopf differential (they are harmonic, so have holomorphic Hopf differential, which can only be identically zero on $S^2$) we see that 
in the setting of Theorem \ref{existing_theory} the 
strong $W_{loc}^{1,2}$ convergence of \eqref{eq:strong-conv-bubble-tree} and the fact that 
the bubbles separate as described in \eqref{eq:bubbles-seperate},
tell us that 
$$\Phi(u_i)\to \Phi(u_\infty) \text{ in } L^1_{loc}(\Upsilon).$$
In case of convergence 
to a bubble branch as in Definition \ref{bb_def}, the map $u_\infty$ also has vanishing Hopf differential (for the same reason) so this further reduces to 
$$
\Phi(u_i)\to 0 \text{ in } L^1_{loc}(\R\times S^1).
$$
\end{rmk}

 \begin{proof}[Proof of Lemma \ref{cor:new}]
 Writing the energy on a connecting cylinder as in \eqref{eq:energy-loss-without-norm}
 and combining this with the estimate \eqref{claim1} on the angular energy as well as the assumption \eqref{ass:almost-harmonic}
immediately implies that  for $m\in\{1,\ldots,\bar m\}$
\beq \label{eq:energy-cyl-proof}
\lim_{\la\to\infty}\limsup_{i\to\infty} \bigg\vert E(u_i;
\Cyl(s_i^{m-1}+\la,s_i^m-\la)) 
-\frac12\int_{\Cyl(s_i^{m-1}+\lambda,s_i^m-\lambda)}(\abs{\partial_s u_i}^2-\abs{\partial_\theta u_i}^2) \,d\theta ds \bigg\vert=0.\eeq

We then note that Remark \ref{rem:conv-Hopf}, applied to the shifted maps 
 $u_i^m(s,\theta)=u_i(s_i^m+s,\theta)$, $m\in \{1,\ldots, \bar m-1\}$,
yields that for every $\lambda>0$
 $$\limsup_{i\to\infty}\int_{\Cyl(s_i^{m}-\lambda,s_i^m+\lambda)}
 \big\vert\abs{\partial_s u_i}^2-\abs{\partial_\theta u_i}^2\big\vert \,d\theta ds\leq 
 \frac12\norm{\Phi(u_i^m)}_{L^1(\Cyl_\lambda)}\to 0.$$
Combined with \eqref{eq:energy-cyl-proof} this gives the claim.
 \end{proof}

The key step needed to derive Theorem \ref{mainthm} from Theorem \ref{new-thm} is to use the Poincar\'e inequality for quadratic differentials to get control on the Hopf differential.

\begin{lemma}
\label{hopf_decay_lem}
In the setting of Theorem \ref{mainthm}, the Hopf differential decays according to
\beqs
\norm{\Phi(u_i,g_i)}_{L^1(M,g_i)} \to 0,\eeqs 
as $i\to\infty$.
\end{lemma}

\begin{proof}
The Poincar\'{e} estimate for quadratic differentials \cite{RT2}
states that for any quadratic differential $\Phi$ on the domain $(M,g)$, and in particular for the Hopf differential $\Phi$, we have
\begin{equation}
\label{poincare}
\|\Phi-P_g(\Phi)\|_{L^1}\leq C\|\overline\partial\Phi\|_{L^1},
\end{equation}
where $C$ depends only on the genus $\gamma\geq 2$ of $M$ and is thus in particular \emph{independent} of $g$.
By \eqref{ass:almost-minimal}, as the area of $(M,g_i)$ is fixed, we know that
$$\|P_{g_i}(\Phi(u_i,g_i))\|_{L^1}\to 0,$$
and by direct computation (see e.g. \cite[Lemma 3.2]{RT}) we know that
$$\|\overline\partial\Phi(u_i,g_i)\|_{L^1}\leq CE_0^\half\|\tau_{g_i}(u_i)\|_{L^2},$$
where $E_0$ is an upper bound on the energies $E(u_i,g_i)$. 
Therefore, by \eqref{ass:almost-minimal} we find that 
$$\|\overline\partial\Phi(u_i,g_i)\|_{L^1}\to 0,$$
and we conclude from \eqref{poincare} that
$$\norm{\Phi(u_i,g_i)}_{L^1(M,g_i)} \to 0,$$
as required.
\end{proof}

Based on Lemma \ref{hopf_decay_lem} and Theorem \ref{new-thm} we can now give the

\begin{proof}[Proof of Theorem \ref{mainthm}]
To derive Theorem \ref{mainthm} from Theorem \ref{new-thm}, we want to view the restriction of 
the maps $u_i$ to the collars $\Col(\ell_i)$ as maps from \textit{euclidean} cylinders $\Cyl_{X_i}$, $X_i=X(\ell_i)$,
which are almost harmonic (with respect to $g_0$).

We first remark that $E(u_i; \Cyl_{X_i})$ is bounded uniformly thanks to the conformal invariance of the energy and the assumed uniform bound on $E(u_i,g_i)$.

We then note that the conformal factors of the metrics 
$\rho^2(s)(ds^2+d\theta^2)$ of
the hyperbolic collars $(\Col(\ell),\rho^2 g_0)$, $\ell\in (0,2\arsinh(1))$, described in Lemma \ref{lemma:collar} are  bounded uniformly by   
$$
\rho(s)\leq \rho(X(\ell))=\frac{\ell}{2\pi\tanh\frac{\ell}{2}}
\leq \frac{\sqrt{2}\arsinh(1)}{\pi}\leq 1.
$$
Given that the norm of the tension scales as 
\beq
\label{3andhalf}
\norm{\tau_g(u)}_{L^2(\Cyl,g)}=\norm{\rho^{-1}\tau(u)}_{L^2(\Cyl)}
\eeq
under a conformal change of the metric $g=\rho^2 g_0$,
we thus obtain from \eqref{ass:almost-minimal} that 
$$
\norm{\tau(u_i)}_{L^2(\Cyl_{X_i})}\leq
\norm{\tau_{g_i}(u_i)}_{L^2(\Col(\ell_i),g_i)}\to 0.
$$

We furthermore note that the $L^1$-norm of quadratic differentials is invariant under conformal changes of metric, compare \eqref{eq:L1-Phi-scaling}, 
and that the Hopf-differential depends only on the conformal structure.
Lemma \ref{hopf_decay_lem} thus yields
$$
\norm{\Phi(u_i)}_{L^1(\Cyl_{X_i})}=\norm{\Phi(u_i,g_i)}_{L^1(\Col(\ell_i),g_i)}\to 0.$$

Consequently all assumptions of  Theorem \ref{new-thm}, including \eqref{ass:Hopf-0}, are satisfied and 
Theorem \ref{mainthm} follows.
\end{proof}
\begin{proof}[Proof of Theorem \ref{energy_id_thm}]
Continuing on from Remarks \ref{rmk:first-part-proof} and \ref{collar_intro_rmk} 
it remains to analyse the energy on the degenerating collars
$\Col(\ell_i^j)$. After passing to a subsequence, Theorem \ref{mainthm} gives convergence to a full bubble branch on each of these collars so that
the energy on $\de\thin(M,g_i)
=\bigcup_j \Col_i^{\de, j}$ satisfies
$$\lim\limits_{\de\downarrow 0}\lim_{i\to \infty}\ E(u_i; \de\thin(M,g_i))=\lim\limits_{\de\downarrow 0}\lim_{i\to \infty}\sum_j E(u_i; \Cyl(-X(\ell_i^j)+\la_{\de}(\ell_i^{j}), X(\ell_i^j)-\la_{\de}(\ell_i^{j}))=\sum_k E(\Om^k).$$
Here we use that $\la_{\de}(\ell_i^j):=X(\ell_i^j)-X_{\de}(\ell_i^j)\geq \frac{\pi}{\de}-C\to \infty$ as $\de\to 0$, compare \eqref{subcyl} and \cite[Prop. A.2]{RT2}, and we denote by $\{\Om_k\}$ the collection of all bubbles developing on the degenerating collars.
\end{proof}

\begin{proof}[Proof of Theorem \ref{thm:torus}]
Let $(u_i,g_i)$ be as in the theorem. 
As before we will think of $(T^2,g_i)$ as a quotient of $(\R\times S^1,\frac1{2\pi b_i} g_0)$ where $(s,\theta)\sim (s+b_i,\theta+a_i)$ are identified and where
we lift the map $u_i$ to the whole cylinder $\R\times S^1$ without changing the notation. Note that $\ell(g_i)\to 0 $ implies $b_i\to \infty$.

Since the maps $u_i$ satisfy \eqref{ass:almost-minimal} we have, by 
\eqref{3andhalf}, 
$$
\norm{\tau(u_i)}_{L^2(\Cyl(0,b_i))}
=\left(\frac1{2\pi b_i}\right)^\frac{1}{2} \norm{\tau_{g_i}(u_i)}_{L^2(T^2,g_i)}\to 0.
$$

The main difference compared with the proof of Theorem \ref{mainthm} is that we are not allowed to appeal to the Poincar\'e inequality since
\eqref{poincare}
is not valid 
 with a uniform constant for tori.

Instead 
we shall use the fact that 
all holomorphic quadratic differentials on $T^2$
are of the form $c\cdot dz^2, z=s+i\theta,  c\in\C$, so that
the projection of the  Hopf differential is 
\beq  
P_g(\Phi)= \bigg(\frac{1}{2\pi b_i}\int_{\Cyl(0,b_i)} \abs{\partial_s u_i}^2-\abs{\partial_\th u_i}^2-2i\langle \partial_\th u_i,\partial_s u_i\rangle d\theta ds\bigg)\cdot dz^2.  \label{eq:Hopf-diff-torus} \eeq

As the tori $(T^2,g_i)$ have unit area,
we furthermore know that 
$\norm{dz^2}_{L^2(T^2, g_i)}=\abs{dz^2}_{g_i} =4\pi b_i$, by \eqref{eq:normalisation-dz2}. Thus \eqref{ass:almost-minimal} implies that 
\beq \label{est:integral-phi-torus}
\int_{\Cyl(0, b_i)} \abs{\partial_s u_i}^2-\abs{\partial_\theta u_i}^2 d\theta ds \to 0 \text{ as } i\to\infty.
\eeq 
Let us first assume that, after passing to a subsequence, we can find a sequence $s_i^0\in [0,b_i)$ 
so that the restrictions of the shifted maps $ u_i(s_i^0+\cdot,\cdot)$ to $\Cyl_{\frac{b_i}{2}}$ converge to a \textit{nontrivial} bubble branch in the sense of Definition \ref{bb_def}.
We can then analyse the restriction of the maps $ u_i(s_i^0+\tfrac12 b_i+\cdot ,\cdot)$ to the cylinder $\Cyl_{\frac{b_i}{2}}$ with Theorem \ref{new-thm} and, after passing to a further subsequence, obtain that the maps $u_i$ converge  to a nontrivial full bubble branch as described in Definition \ref{def:fbb-torus}.

To compute the loss of energy on the connecting cylinders we can think of $u_i$ as a periodic map from $\Cyl_{s_i^0-b_i,s_i^0+2b_i}$. In the  convergence to 
a full bubble branch for these extended maps also the cylinders $\Cyl(s_i^{0},s_i^1)$ and 
$\Cyl(s_i^{\bar m-1},s_i^{\bar m})$, $s_i^{\bar m}=s_i^0+b_i$, obtained above appear as connecting cylinders between two bubble branches rather than cylinders 
connecting a bubble branch with the end of the domain. As such  
Lemma \ref{cor:new} implies that
\beqs 
\lim_{\la\to\infty}\limsup_{i\to\infty} 
\bigg|E(u_i;
\Cyl(s_i^{m-1}+\la,s_i^m-\la)) -\frac12\int_{\Cyl(s_i^{m-1},s_i^m)} (\abs{\partial_s u_i}^2-\abs{\partial_\theta u_i}^2) \,d\theta ds \bigg|=0
\eeqs
is valid not only for $m=2,\ldots, \bar m-1$ but indeed for $m=1,\ldots \bar m$.
The total amount of energy lost on the connecting cylinders is thus
 \beqas 
 \lim_{\lambda\to\infty}\limsup_{i\to \infty}\sum_{m=1}^{\bar m} E(u_i;\Cyl(s_i^{m-1}+\lambda,s_i^{m}-\lambda))
=
 \limsup_{i\to\infty} \half\int _{\Cyl(0,b_i)}\abs{\partial_s u_i}^2-\abs{\partial_\th u_i}^2 \, dsd\th=0.
\eeqas

Assume instead that we cannot find, even for a subsequence, numbers $s_i^0$ so that $ u_i(s_i^0+\cdot,\cdot)$ converges to a \textit{nontrivial} bubble branch. In this case 
we apply Theorem \ref{new-thm} 
twice on   
$ \Cyl_{\half b_i}$, once 
for $u_i$ itself and once for the shifted map $\hat u_i:=u_i(\cdot+ \half b_i,\cdot)$. The assumption means that the set of points $s_i^m$, $m=1,..,\bar m-1$, where nontrivial bubble branches develop, is empty, i.e. $\bar m=1$. 
Part 2 of Theorem \ref{new-thm} thus implies that
the whole torus is mapped close to an $i$-dependent curve as described in the theorem. 
But  Theorem \ref{existing_theory} also implies that, after passing to a subsequence, the maps $u_i$ converge to a bubble branch, which must be trivial by assumption, and thus that  
$E(u_i;\Cyl_\lambda)\to 0$ for every $\lambda>0$. Combined with Lemma \ref{cor:new} and \eqref{est:integral-phi-torus} we thus obtain
\beqas 
\limsup_{i\to \infty} E(u_i,g_i)=&\lim_{\la\to \infty}\limsup_{i\to\infty}E(u_i; \Cyl(\lambda, b_i-\lambda))
=\frac12\limsup_{i\to\infty}\int_{\Cyl(0, b_i)} \abs{\partial_s u_i}^2-\abs{\partial_\theta u_i}^2 d\theta ds=0
\eeqas
as claimed. 
\end{proof}

\begin{rmk}\label{rem:Ding-Li-Liu}
When we apply Theorem \ref{thm:torus} to
the flow \eqref{flow}, using Proposition \ref{get_ti}, we obtain 
a no-loss-of-energy result from \eqref{torus_no_loss_of_energy} for $M$ a torus.
An assertion of this form
was already made in \cite{Ding-Li-Liu}, where a claim is made that the Hopf differential converges to zero in $L^1$, based on the knowledge that its \emph{`average'} converges to zero.
(Note that taking our viewpoint, the fact that the average of the Hopf differential converges to zero would be interpreted 
as the projection of the Hopf differential onto the space of holomorphic quadratic differentials converging to zero.)
However, our take on the genus one case, in the absence of a uniform Poincar\'e estimate available in the higher genus case, is that we can only deduce that the $L^1$ norm of the Hopf differential converges to zero as a \emph{consequence} of no-loss-of-energy, and along the way we need to appeal to the $L^1$ smallness of the Hopf differential of bubble branches from Remark
\ref{rem:conv-Hopf}.
\end{rmk}

\section{Construction of a nontrivial neck}
\label{neck_example_sect}

The main purpose of this section is to prove Theorem \ref{neck_thm}, but we first record the following elementary computations.
\begin{proof}[Proof of Proposition \ref{nongeod}]
To ease notation, we drop all subscripts $i$ for the following computations. We also simplify matters by embedding $(N,G)$ isometrically in some Euclidean space and composing $u$ with that embedding.
The energy is conformally invariant, thus we calculate with respect to the flat metric 
$$
E(u;\Cyl_{X}) = \frac{1}{2} \int_{S^1}\int_{-X}^{X} |u_s|^2 ds d\theta \leq  \frac{CL^2}{X}.
$$
Next we recall that the volume element on the hyperbolic cylinder is given by $\rho^2 ds d\theta$, and compute using \eqref{3andhalf}
$$
\|\tau_g(u)\|_{L^2(\Cyl_X,g)}^2
=\|\rho^{-1}\tau(u)\|_{L^2(\Cyl_X)}^2
\leq C\int_{-X}^X \rho^{-2}|u_{ss}|^2ds
\leq \frac{CL^4}{X^4}\int_{-X}^X \rho^{-2}ds
\leq \frac{CL^4}{X^4}X\ell^{-2}\leq \frac{CL^4}{X}
$$
because $\ell^{-1}\leq CX$ on a collar.
We now compute the $L^2$-norm of $\Phi(u,g)$.  With $z = s + i \theta$, $\Phi(u,g) = |u_s|^2 dz^2$. Recalling that $|dz^2|_g = 2 \rho^{-2}$ (see \ref{eq:normalisation-dz2}), we find
$$
\norm{\Phi(u,g)}_{L^2(\Cyl_X,g)}^2 =
\int_{\Cyl_X} |u_s|^4 4\rho^{-4}\rho^2ds\,d\th
\leq C\frac{L^4}{X^4}\int_{-X}^X  \rho^{-2}ds
\leq C\frac{L^4}{X}
$$
as above.
\end{proof}

The remainder of this section is devoted to the proof of Theorem \ref{neck_thm}, constructing a flow that develops a nontrivial neck. 
We opt for a general approach, although essentially explicit constructions are also possible.
To this end, consider  any closed oriented surface $M$ of genus at least 2, and 
take the target $N$ to be $S^1$. Choose a smooth initial map $u_0:M \to S^1$ that maps some closed loop $\alpha$ on $M$ exactly once around $S^1$, and take any hyperbolic metric $g_0$ on $M$.
We claim that the subsequent flow \eqref{flow} develops a nontrivial neck.  

The first key point is that since $S^1$ has nonpositive sectional curvature, the regularity theory from \cite[Theorem 1.1, Theorem 1.2]{RTnonpositive} applies, so the flow exists for all time.

The second key point is that because the target is $S^1$, there do not exist any branched minimal immersions, except if one allows constant maps.
If no collar degenerated in this flow, i.e. if there were a uniform positive lower bound for the lengths of all closed geodesics in $(M,g(t))$,
then by the results in \cite{RT}, the map $u_0$ would be homotopic to the constant map, which is false by hypothesis.

Therefore there are degenerating collars, and we can analyse them with 
Theorems \ref{thm:asymptotics1_new} and \ref{mainthm} (using Proposition \ref{get_ti});
we next demonstrate that a nontrivial neck as described in Theorem \ref{neck_thm} forms. If not, the maps from each degenerating collar would become close to  constant maps, for a subsequence. By \cite[Theorem 1.1]{RTZ} this would imply that $u_0$ would be decomposed into constant maps, and thus in particular it would be homotopic to a constant map, which again is false by hypothesis.

We have proved that our flow develops a neck in the sense that \eqref{neck_develops} holds for some degenerating collar, and some $m$.
By Theorem \ref{mainthm} the image of the subcollar 
$\Cyl(s_i^{m-1}+\la,s_i^m-\la)$
will be close to a curve for each $i$.
However, the limiting endpoints \eqref{pplus} and \eqref{pminus} of the curves would not in general be distinct, i.e.
\eqref{ps_unequal} would fail in general.

To make a construction in which \eqref{ps_unequal} must hold for some collar and some $m$, it suffices to adjust our construction so that again all extracted branched minimal immersions must be constant, but so that the union of the images is not just one point. By our theory, the connecting cylinders will thus be mapped close to curves connecting these distinct image points so a nontrivial neck with distinct end points must develop. 

To achieve this, we will lift the whole flow to a finite cover $\overline{M}$ of $M$. Given such a cover, we need to check that the lifted flow still satisfies \eqref{flow}. Locally, the lifting does not affect the tension field of $u$, so the first equation in \eqref{flow} does not present any issues. However, the projection operator $P_g$ will be different on the cover $\overline{M}$, as new holomorphic quadratic differentials are introduced in addition to the lifts of the original holomorphic quadratic differentials. But as we check in the following lemma, it turns out that these new directions are necessarily $L^2$-orthogonal to the (lifted) Hopf differential $\Phi$, and indeed to any quadratic differential on the cover that is obtained by lifting a quadratic differential on $M$.

\begin{lemma} 
\label{lem:lifting}
Let $(\overline{M} ,\overline{g})$ be a smooth cover of $(M,g)$ with a finite group of deck transformations $G$ associated with it. Given a holomorphic quadratic differential $\Psi$ on $\overline{M}$, define
	$$\Proj(\Psi) = \frac{1}{|G|} \sum\limits_{\mathfrak{g} \in G} \mathfrak{g}^*(\Psi).$$
Note that this defines a projection operator, and the image can be identified with the space of holomorphic quadratic differentials on $M$. Then for any $\Psi \in \ker(\Proj)$ and $\Phi$ a quadratic differential on $\overline{M}$ arising via lifting a quadratic differential on $M$, thus invariant under $G$, we have $\langle \Psi, \Phi \rangle_{L^2} = 0$.
\end{lemma}
\begin{proof}
Let $\Psi \in \ker(\Proj)$, then $\sum\limits_{\mathfrak{g} \in G}{\mathfrak{g}^*(\Psi)} = 0$. Furthermore, $$\langle\Psi, \Phi \rangle = \langle \mathfrak{g}^*(\Psi),\mathfrak{g}^*(\Phi) \rangle = \langle \mathfrak{g}^*(\Psi),\Phi \rangle$$ for any $\mathfrak{g} \in G$. Summing, we obtain
$$ 0= \left\langle \sum_{\mathfrak{g} \in G} \mathfrak{g}^*(\Psi), \Phi \right\rangle = |G| \langle \Psi, \Phi \rangle $$ 
as required.
\end{proof}

As a consequence of this lemma, given a suitable cover $\overline{M}$ of $M$, we can  lift a flow on $M$ to $\overline M$ and the flow equations \eqref{flow} will still hold. 

Returning to our construction,
we may assume that all the branched minimal immersions 
are mapping to the same limit point $p$ (or we are done already).
The aim is now to use a lifting construction, justified by the above, to obtain a lifted flow with 
the images of the corresponding branched minimal immersions being the
two \emph{different} lifts of $p$ in a double cover of the target.

To this end, fix an arbitrary base point $x_0$ on M. Now consider the index 2 subgroup $H$ of $\pi_1(M,x_0)$ consisting of loops whose images under the initial map $u_0$ go round the target $S^1$ an even number of times.
We can pass to a (double) cover $q:\overline{M} \to M$ of the domain satisfying $q_*\left(\pi_1\left( \overline{M},\overline{x}_0 \right) \right) = H$ (see e.g. \cite[Prop. 1.36]{Hatcher})
and lift $u_0$ to a map
$\overline{u}_0: \overline{M} \to S^1$. By the choice of $H$ we can further lift $\overline{u}_0$ to map into a (connected) double cover of the target $S^1$ (e.g. \cite[Prop. 1.33]{Hatcher}).
Using Lemma \ref{lem:lifting}, we can in fact lift the whole flow like this to give a new solution, and analyse it with Theorem \ref{mainthm} for a subsequence of the times $t_i$ at which we analysed the flow on $M$.

It suffices to show that the images of the branched minimal immersions we can construct from the lifted flow consist of \emph{both} of the lifts of $p$, not just one. These distinct points can then only be connected by nontrivial necks.

To see this, note that from the analysis of the original flow on $M$ with Theorem 1.1 from \cite{RTZ} (i.e. with Proposition \ref{get_ti} and Theorem \ref{thm:asymptotics1_new}), we can find some $\delta > 0$ sufficiently small such that for sufficiently large $i$, the $\delta$-thin part of $(M,g(t_i))$ will  consist of a (disjoint) union of (sub)collars that eventually degenerate.
For large enough $i$, the image of the $\delta$-thick part of $(M,g(t_i))$ will be contained in a small neighbourhood of $p$. For each such large $i$, we pick a point $y_i$ in the $\delta$-thick part, and deform $\al$ to pass through $y_i$.
We view $\al$ then as a path that starts and ends at $y_i$, and by assumption, the composition $u_0\circ\al$ takes us 
exactly once round the target $S^1$.
In particular, as we pass once round the lift of $\al$, we move from one lift of $y_i$ to the other, and the flow map moves from being close to one lift of $p$ to being close to the other lift.

In particular, the branched minimal immersions in the lifted picture map to both lifts of $p$ as required.

\appendix
\section{Appendix}
We will need Keen's `Collar lemma' throughout the paper.

\begin{lemma}[\cite{randol}] \label{lemma:collar}
Let $(M,g)$ be a closed hyperbolic surface and let $\si$ be a simple closed geodesic of length $\ell$. Then there is a neighbourhood around $\si$, a so-called collar, which is isometric to the 
cylinder 
$\Col(\ell):=(-X(\ell),X(\ell))\times S^1$
equipped with the metric $\rho^2(s)(ds^2+d\theta^2)$ where 
$$\rho(s)=\frac{\ell}{2\pi \cos(\frac{\ell s}{2\pi})} 
\qquad\text{ and }\qquad  
X(\ell)=\frac{2\pi}{\ell}\left(\frac\pi2-\arctan\left(\sinh\left(\frac{\ell}{2}\right)\right) \right).$$ 
The geodesic $\si$  corresponds to the circle 
$\{s=0\}\subset \Col(\ell)$. 
\end{lemma}

For $\de\in (0,\arsinh(1))$, the $\de\thin$ part of a collar is given by the subcylinder 
\beq
\label{subcyl}
(-X_\de(\ell), X_\de(\ell)) \times S^1 \subseteq \Col(\ell),    
\qquad\text{where}\qquad
X_\de(\ell)=  \frac{2\pi}{\ell}\left(\frac{\pi}{2}-\arcsin \left(\frac{\sinh(\frac{\ell}{2})}{\sinh \delta}\right) \right)\eeq
 for $\de\geq \ell/2$, respectively zero for smaller values of $\delta$.

To analyse sequences of degenerating hyperbolic surfaces we make repeatedly use of the differential geometric version of the Deligne-Mumford compactness theorem.
\begin{prop} {\rm (Deligne-Mumford compactness, cf. \cite{Hu}.)}  \label{Mumford}
Let $(M,g_{i},c_{i})$ be a sequence of closed hyperbolic Riemann surfaces of genus $\gamma\geq2$ that degenerate in the sense that $\liminf_{i\to \infty}\inj_{g_{i}}M=0$.
Then, after selection of a subsequence,
$(M,g_{i},c_{i})$ converges to a complete hyperbolic punctured Riemann surface $(\Sigma,h,c)$, where $\Si$ is obtained from $M$ by removing a collection 
$\mathscr{E}=\{\sigma^{j}, j=1,...,k\}$ of $k$ pairwise disjoint, homotopically nontrivial, simple closed curves on $M$ and the convergence is as follows:

For each $i$ there exists a collection $\mathscr{E}_{i}=\{\sigma^{j}_{i}, j=1,...,k\}$ of
pairwise disjoint simple closed geodesics on $(M,g_{i},c_{i})$
of length
$\ell(\sigma_{i}^{j})=:\ell_{i}^{j} \rightarrow 0\text{ as }i \rightarrow \infty$,
and a diffeomorphism $F_i:M\to M$ mapping $\si^j$ onto $\si_i^j$,
such that the restriction
$f_i=F_i|_\Si:\Si\rightarrow M\setminus \cup_{j=1}^k\sigma_{i}^{j} $ satisfies
$$(f_i)^{*}g_{i} \rightarrow h \text{ and } (f_i)^{*}c_{i}\to c \text{ in } C_{loc}^{\infty}\text{ on }\Sigma.$$
\end{prop}

Finally, for metrics of the form $g=\xi^2 g_0$, $\xi:\Cyl\to \R^+$ any function  on a cylinder $\Cyl=\Cyl(s_1,s_2)$, we have
\beq \label{eq:normalisation-dz2}
\abs{dz^2}_{g}=\xi^{-2} \abs{dz^2}_{g_0}=2\xi^{-2}. 
\eeq
Thus the $L^1$ norm of a quadratic differential $\Psi=\psi dz^2$ is independent of the conformal factor:
\beq 
\label{eq:L1-Phi-scaling} \norm{\Psi}_{L^1(\Cyl,g)}=\norm{\Psi}_{L^1(\Cyl,g_0)}=2\int_\Cyl \abs{\psi} dsd\theta.
\eeq

{\sc TH\&PT: Mathematics Institute, University of Warwick, Coventry,
CV4 7AL, UK}

{\sc MR: 
Mathematical Institute, University of Oxford, Oxford, OX2 6GG, UK}

\end{document}